\documentclass[reqno]{amsart}
\usepackage{graphicx}
\usepackage{amsthm,amsmath,amssymb}
\usepackage{mathrsfs}
\usepackage{enumerate}
\usepackage{amsaddr}
\usepackage{subfigure}
\theoremstyle{plain}
\newtheorem{theorem}{Theorem}

\newtheorem{corollary}{Corollary}
\newtheorem{proposition}{Proposition}

\newtheorem{observation}{Observation}

\theoremstyle{definition}
\newtheorem{definition}{Definition}
\newtheorem{notation}{Notation}
\newtheorem{example}{Example}

\theoremstyle{remark}
\newtheorem{remark}{Remark}

\numberwithin{equation}{section}

\newcommand{\abs}[1]{\left\vert#1\right\vert}

\newcommand{\A}{\mathfrak A}
\newcommand{\Text}[1]{\text{\textnormal{#1}}}
\newcommand{\Aut}{\Text{Aut}}

\begin{document}

\title{Computing the autotopy group of a Latin square by cycle structure}%
\author{Daniel Kotlar}%
\address{Computer Science Department, Tel-Hai College, Upper galilee, Israel}%
\email{dannykot@telhai.ac.il}%
\setlength{\parskip}{0.075in}
\setlength\parindent{0pt}

\begin{abstract}
An algorithm that uses the cycle structure of the rows, or the columns, of a Latin square to compute its autotopy group is introduced. As a result, a bound for the size of the autotopy group is obtained. This bound is used to show that the computation time for the autotopy group of Latin squares that have two rows or two columns that map from one to the other by a permutation which decomposes into a bounded number of disjoint cycles, is polynomial in the order $n$.
\end{abstract}
\maketitle
\section{Introduction}\label{sec1}
For any positive integer $n$, a \emph{Latin square} of order $n$ is an $n\times n$ array of numbers in $[n]=\{1,\ldots,n\}$, so that each row and each column of $L$ is a permutation of $[n]$. A Latin square is called \emph{reduced} if its first row and first column are equal to the identity permutation.
A \emph{line} in $L$ is either a row or a column.

Let $S_n$ be the symmetric group of permutations of $[n]$. An \emph{isotopism} is a triple $(\alpha,\beta,\gamma)\in S_n^3$ that acts on the set of Latin squares of order $n$ by permuting the set of rows of a Latin square by $\alpha$, permuting the set of columns by $\beta$, and permuting the symbols by $\gamma$.
An \emph{autotopism} of a Latin square $L$ is an isotopism $\Theta$ such that $\Theta(L)=L$. The \emph{autotopy group} of $L$, denoted $\A(L)$, is the group of autotopisms of $L$.
Two Latin squares are called \emph{isotopic} if there is an isotopism that transforms one to the other. If $L$ and $L'$ are isotopic, say $\Theta(L)=L'$, then their autotopy groups are related: $\A(L')=\Theta\A(L)\Theta^{-1}$. Since every Latin square is isotopic to a reduced Latin square, we can study the structure and size of autotopy groups of general Latin squares by exploring autotopy groups of reduced Latin squares. For further knowledge about isotopisms and autotopisms the reader is referred to \cite{Denes74, Janssen95, MckayMM07, MckayWan05, Sade68}, among many others.

Autotopy groups, as well as ways to compute them, have been the subject of many studies. McKay \cite{McKay, MckayMM07} introduced ``nauty'', an algorithm for computing the symmetry groups of a graph, and used it to compute the various symmetry group of a Latin square, including the autotopy group, by mapping the Latin square to a certain graph. For some otherwise tough classes of Latin squares, ``nauty'' can be accelerated by using vertex invariants such as the ``train'' introduced
by Wanless \cite{wan05}.

Although the proportion of Latin squares of order $n$ which have non-trivial autotopy group tends quickly to zero as $n$ grows (McKay and Wanless \cite{MckayWan05}), some special Latin squares of any order may have a large autotopy group (see Wanless \cite{wan05}). For example, if $L$ is the Cayley table of a group $G$ of order $n$ then
\begin{equation}\label{eq1:0}
    |\A(L)|=n^2|\Aut(G)|
\end{equation}
(as mentioned in \cite{Bailey82} and \cite{Bro2012} and shown in \cite{Sade68} and \cite{Sch30}). Browning, Stones and Wanless \cite{Bro2012} set the following general bound for the size of the autotopy group of a Latin square $L$ of order $n$:
\begin{equation}\label{eq1:1}
    |\A(L)|\leq n^2 \prod_{t=1}^{\lfloor\log_2n\rfloor}\left(n-2^{t-1}\right).
\end{equation}
Given a specific Latin square $L$, other bounds can be computed using easily computable features of $L$. For example, viewing the rows and columns of a Latin square as permutations in $S_n$, and assuming that $L$ has $k$ rows of one parity (even or odd) and $n-k$ rows of the opposite parity, it was shown in \cite{kot12} that
\begin{equation}\label{eq1:2}
    |\A(L)|\leq n(n-k)!k!.
\end{equation}
Any permutation in $\sigma\in S_n$ can be written as a product of disjoint cycles - the \emph{cycle representation} of $\sigma$. The lengths of these cycles define a partition of $n$ called the \emph{cycle structure} of $\sigma$. The cycle structures of permutations have been considered in the context of Latin squares in different aspects. Cavenagh, Greenhill and Wanless \cite{Cav08} considered the cycle structure of the permutation that transforms one row of a Latin square to another row.
Falc\'{o}n \cite{Fal09} and Stones, Vojt\v{e}chovsk\'{y} and Wanless \cite{Stones11} considered the cycle structure of the permutations $\alpha$, $\beta$ and $\gamma$ in a given isotopism $\Theta=(\alpha, \beta, \gamma)$, in order to derive information on Latin squares for which $\Theta$ is an autotopism.
Ga{\l}uszka \cite{galuszka2008codes} used the cycle structure of rows of Latin squares, viewed as Cayley tables of groupoids, in order to study the quasigroup structure of the groupoids. In \cite{kot12} the author used the cycle structure of the rows of a reduced Latin square to obtain a bound on the size of the autotopy group:
let $L$ be a reduced Latin square and suppose $(\lambda_{1},\lambda_{2},\ldots,\lambda_{s})$ is a partition of $n$ by the different cycle structures of the rows, then
\begin{equation}\label{eq1:3}
\abs{\A(L)}\leq n^2\prod_{i=1}^s\lambda_{i}!.
\end{equation}

In this paper we introduce an algorithm that uses the cycle structure of the rows of $L$ for finding $\A(L)$. The algorithm yields a bound for $|\A(L)|$ that involves cycle structure data. It is shown that for the family of reduced Latin squares that have two rows or two columns that map from one to the other by a permutation having a bounded number of disjoint cycles, the autotopy group can be computed in polynomial time in $n$.

\begin{notation}\label{not0:1}
For a Latin square $L$ of order $n$ let $\{\sigma_i\}_{i=1}^n$ be the rows of $L$, viewed as permutations, and let $\{\pi_i\}_{i=1}^n$ be the columns of $L$.
\end{notation}

\textbf{Convention:} When viewing a row or a column of a Latin square as a permutation $\sigma\in S_n$, it is understood that the number $i$ appearing in the $j$th place of $\sigma$ signifies that $\sigma(j)=i$.

\section{Preliminary Results and Notation}
In this section we mention useful known results and introduce some notation.
\begin{notation}\label{not1:1}
Let $L$ be a reduced Latin square. For any permutation $\alpha\in S_n$ and any column $\pi_j$ of $L$ let $\Theta_{\alpha,j}$ denote the isotopism $(\alpha, \alpha\pi_j^{-1}\sigma_{\alpha^{-1}(1)}, \alpha\pi_j^{-1})$.
\end{notation}
Proposition~\ref{prop1:1} appears in different formulations in \cite{StWan12} and \cite{kot12}. It describes the $n\cdot n!$ isotopisms that map a given reduced Latin square to a reduced Latin square.
\begin{proposition}\label{prop1:1}
Let $L$ be a reduced Latin square.
\begin{enumerate}
  \item [\Text{(i)}]
  For any permutation $\alpha\in S_n$ and any column $\pi_j$ of $L$, the Latin square $\Theta_{\alpha,j}(L)$ is reduced.
  \item [\Text{(ii)}]
  If $\Theta$ is an isotopism such that $\Theta(L)$ is reduced then $\Theta=\Theta_{\alpha,j}$ for some $\alpha\in S_n$ and some column $\pi_j$ of $L$.
\end{enumerate}
\end{proposition}

The next proposition from \cite{kot12} describes the effect of applying $\Theta_{\alpha,j}$ on a single row of a Latin square (the slight difference from the original in \cite{kot12} is due to the different convention used there to interpret a line in a Latin square as a permutation).

\begin{proposition}\label{prop1:2}
Let $L$ be a reduced Latin square of order $n$ with row permutations $\{\sigma_i\}_{i=1}^n$ and column permutations $\{\pi_j\}_{j=1}^n$. Let $\Theta_{\alpha,j}$ be an isotopism as defined in Notation~\ref{not1:1}. Let $L'=\Theta_{\alpha,j}(L)$ and let $\{\sigma'_i\}_{i=1}^n$ be the rows of $L'$. Then, for all $i=1,\ldots,n$,
\begin{equation}\label{eq6:01}
\sigma'_i=\alpha\pi_j^{-1}\sigma_{\alpha^{-1}(i)}\sigma_{\alpha^{-1}(1)}^{-1}\pi_j\alpha^{-1}.
\end{equation}
In particular, if $\Theta_{\alpha,j}\in\A(L)$, then for all $i=1,\ldots,n$,
\begin{equation}\label{eq6:02}
\sigma_i=\alpha\pi_j^{-1}\sigma_{\alpha^{-1}(i)}\sigma_{\alpha^{-1}(1)}^{-1}\pi_j\alpha^{-1}.
\end{equation}
\end{proposition}

\begin{notation}\label{not1:5}
For a permutation $\sigma\in S_n$ let $\nu(\sigma)$ denote the number of cycles in the cycle representation of $\sigma$ (including cycles of length 1). For a Latin square $L$ of order $n$ let $\nu(L)=\min_i \nu(\sigma_i)$.
\end{notation}

\begin{notation}\label{not2}
Let $L$ be a reduced Latin square of order $n$. For any $k\in[n]$ let $\lambda(L,k)$ denote the number of rows with the same cycle structure as the row $\sigma_k$. Let $\lambda(L)=\max_k\lambda(L,k)$.
\end{notation}

\begin{example}\label{ex0:1}
Consider the following reduced Latin square $L$ of order 8:
\begin{equation*}
\begin{array}{|c|c|c|c|c|c|c|c|}
\hline
 1 & 2 & 3 & 4 & 5 & 6 & 7 & 8\\
\hline
 2 & 1 & 4 & 6 & 8 & 7 & 5 & 3\\
\hline
 3 & 4 & 1 & 2 & 6 & 5 & 8 & 7\\
\hline
 4 & 5 & 8 & 7 & 3 & 2 & 1 & 6\\
\hline
 5 & 7 & 6 & 1 & 4 & 8 & 3 & 2\\
\hline
 6 & 8 & 5 & 3 & 7 & 1 & 2 & 4\\
\hline
 7 & 6 & 2 & 8 & 1 & 3 & 4 & 5\\
\hline
 8 & 3 & 7 & 5 & 2 & 4 & 6 & 1\\
\hline
\end{array}
\end{equation*}
The cycle representations of the rows of $L$, grouped by cycle structure, are:
\begin{equation}\label{eq:cs1}
\begin{array}{c l}
\Text{Row} & \Text{Cycle representation}\\
\hline
\hline
1   &   (1)(2)(3)(4)(5)(6)(7)(8)\\
\hline
3	&   (1,3)(2,4)(5,6)(7,8)\\
\hline
2	&   (1,2)(3,4,6,7,5,8)\\
6	&   (1,6)(2,8,4,3,5,7)\\
8	&   (1,8)(2,3,7,6,4,5)\\
\hline
4	&   (1,4,7)(2,5,3,8,6)\\
5	&   (1,5,4)(2,7,3,6,8)\\
7	&   (2,6,3)(1,7,4,8,5)\\
\end{array}
\end{equation}
We have $\nu(L)=2$. Also, $\lambda(L,1)=\lambda(L,3)=1$ and $\lambda(L,t)=3$ for $t\ne 1,3$. Thus, $\lambda(L)=3$.
\end{example}

\begin{notation}\label{not3}
Let $L$ be a reduced Latin square of order $n$ with row permutations $\{\sigma_i\}_{i=1}^n$. For any $i,k\in[n]$ denote by $\sigma_{i,k}$ the permutation $\sigma_i\sigma_k^{-1}$. Denote by $\Delta(L)$ the set of integers $k\in[n]$ such that the multiset of cycle structures of $\{\sigma_{i,k}\}_{i=1}^n$ is the same as the multiset of cycle structures of $\{\sigma_i\}_{i=1}^n$. Let $\delta(L)=|\Delta(L)|$.
\end{notation}

\begin{example}
If $L$ is the Cayley table of a finite group of order $n$, then $\{\sigma_{i,k}\}_{i=1}^n=\{\sigma_i\}_{i=1}^n$ for all $k\in[n]$. Thus, $\delta(L)=n$.
\end{example}

\begin{observation}\label{obs0}
Let $L$ be a reduced Latin square of order $n$ and let $\Theta_{\alpha,j}\in \A(L)$. If $k$ is in the $\alpha$-orbit of 1, then $k\in \Delta(L)$.
\end{observation}

\begin{proof}
Suppose $\alpha^t(k)=1$ for some $t$. Since $\A(L)$ is a group, $\Theta_{\alpha,j}^t\in \A(L)$. Note that $\Theta_{\alpha,j}^t=\Theta_{\alpha^t,m}$ for some $m\in[n]$, by Proposition~\ref{prop1:1}. Let $\beta=\alpha^t$. We have $\beta^{-1}(1)=k$. Since $\Theta_{\beta,m}\in \A(L)$ we have $\sigma_i=\beta\pi_m^{-1}\sigma_{\beta^{-1}(i)}\sigma_k^{-1}\pi_m\beta^{-1}$ for all $i=1,\ldots,n$, by (\ref{eq6:02}). Thus, $k\in \Delta(L)$.
\end{proof}

\begin{notation}\label{not4}
Let $L$ be a reduced Latin square of order $n$ and suppose $k\in \Delta(L)$. For any $t\in[n]$ let
\begin{equation*}
    R_k(L,t):=\{i\in[n]: \sigma_t \Text{ and } \sigma_{i,k}\Text{ have the same cycle structure}\}.
\end{equation*}
\end{notation}
The following two observations follow directly from (\ref{eq6:02}):
\begin{observation}\label{obs1}
Let $L$ be a reduced Latin square of order $n$. If $k\in \Delta(L)$, then $R_k(L,1)=\{k\}$ and $|R_k(L,t)|=\lambda(L,t)$ for all $t\in[n]$.
\end{observation}

\begin{observation}\label{obs2}
Let $L$ be a reduced Latin square of order $n$ and let $\Theta_{\alpha,j}\in \A(L)$. If $\alpha(k)=1$, then for any $t\in[n]$,
\begin{equation}\label{eq2:04}
    \alpha^{-1}(t)\in R_k(L,t).
\end{equation}
\end{observation}

\begin{remark}
Since there are $n$ options for $\pi_j$ in Proposition~\ref{prop1:1} and $\delta(L)\le n$ the bound in (\ref{eq1:3}) follows.
\end{remark}

\begin{example}\label{ex1}
The Latin square $L$ in Example~\ref{ex0:1} has $\Delta(L)=\{1,2\}$. The cycle representations of $\{\sigma_{i,2}\}_{i=1}^8$ are:

\begin{equation}\label{eq:cs2}
\begin{array}{c l}
i & \Text{Cycle rep. of }\sigma_{i,2} \\
\hline
\hline
2   &   (1)(2)(3)(4)(5)(6)(7)(8)\\
\hline
8   &   (1,3)(2,8)(4,7)(5,6)\\
\hline
1	 &   (1,2)(3,8,5,7,6,4)\\
3	 &   (1,4)(2,3,7,5,8,6)\\
4	 &   (1,5)(2,4,8,3,6,7)\\
\hline
5   &   (2,5,3)(1,7,8,4,6)\\
6   &   (1,8,7)(2,6,3,4,5)\\
7   &   (1,6,8)(2,7,3,5,4)\\
\end{array}
\end{equation}

We have $R_2(L,1)=\{2\}$, $R_2(L,3)=\{8\}$, $R_2(L,2)=R_2(L,6)=R_2(L,8)=\{1,3,4\}$, and $R_2(L,4)=R_2(L,5)=R_2(L,7)=\{5,6,7\}$.
\end{example}

\section{Computing $\A(L)$}\label{sec2}

Let $L$ be a reduced Latin square and let $\sigma_l$ be a row permutation in $L$ with minimal number of cycles. We shall construct isotopisms $\Theta_{\alpha,j}$ that fix $\sigma_l$. By (\ref{eq6:02}) we have,
\begin{equation}\label{eq3:1}
\sigma_l=\alpha\pi_j^{-1}\sigma_{\alpha^{-1}(l)}\sigma_{\alpha^{-1}(1)}^{-1}\pi_j\alpha^{-1}.
\end{equation}
The construction of $\Theta_{\alpha,j}\in\A(L)$ goes as follows: we first choose $\alpha^{-1}(1)$ and $\alpha^{-1}(l)$. Then choose $j$ and compute $\pi_j^{-1}\sigma_{\alpha^{-1}(l)}\sigma_{\alpha^{-1}(1)}^{-1}\pi_j$. Then we use Equation (\ref{eq3:1}) to determine $\alpha$, and thus we have a candidate $\Theta_{\alpha,j}$. Finally, we check whether $\Theta_{\alpha,j}$ fixes the other rows. Due to the different constraints, most $\Theta_{\alpha,j}$'s will be eliminated along the way and only a small fraction will reach the final stage.

Here is a detailed description of the algorithm:

\textbf{Step 1:}
As already mentioned, We choose a row $\sigma_l$ in $L$ with a minimal number of cycles, i.e. $\nu(L,l)=\nu(L)$. The index $l$ will remain constant throughout the algorithm.

\textbf{Step 2:} We next wish to choose $k=\alpha^{-1}(1)$. By Observation~\ref{obs0}, we must choose $k$ from $\Delta(L)$. We have
  \begin{equation}\label{eq3:2}
    \alpha(k)=1.
 \end{equation}

We define the binary matrix $T(L,k)$ as the $n\times n$ matrix whose $(i,j)$ entry is 1 if and only if $\sigma_{i,k}$ has the same cycle structure as $\sigma_j$. This means that there might be an autotopism $(\alpha,\beta,\gamma)$ of $L$ such that $\alpha(i)=j$, by Observation~\ref{obs2}. Any all-1's generalized diagonal in $T(L,k)$ corresponds to a possible such permutation $\alpha$. (The term \emph{generalized diagonal} of an $n\times n$ matrix refers to a set of $n$ entries belonging to distinct rows and distinct columns.)

\begin{remark}\label{rem2}
If we group the rows $T_i$ of $T(L,k)$ by the cycle structures of $\{\sigma_{i,k}\}_{i=1}^n$ and group the columns $T^j$ of $T(L,k)$ by the cycle structures of $\{\sigma_{j}\}_{j=1}^n$ (as in (\ref{eq:cs1}) and (\ref{eq:cs2})), keeping the original row and column indexes, the resulting matrix consists of all-1's square blocks in the diagonal and 0's elsewhere. For example, the matrices $T(L,1)$ and $T(L,2)$ for the Latin square $L$ in Example~\ref{ex0:1}, after such row and column rearrangements (and omitting the 0's for clarity) are

\resizebox{\linewidth}{!}{
$\begin{array}{cc}
    &
    \begin{array}{cccccccc}{}^1 & {}^3 & {}^2 & {}^6 & {}^8 & {}^4 & {}^5 & {}^7
    \end{array}
  \\
  \begin{array}{r}  {}_1\\{}_3\\{}_2\\{}_6\\{}_8\\{}_4\\{}_5\\{}_7
  \end{array}
  & \left(\begin{array}{cccccccc}
  1& & & & & & &  \\
   &1& & & & & &  \\
   & &1&1&1& & &  \\
   & &1&1&1& & &  \\
   & &1&1&1& & &  \\
   & & & & &1&1&1 \\
   & & & & &1&1&1 \\
   & & & & &1&1&1
  \end{array}\right)
\end{array}
\quad\Text{and}\quad
\begin{array}{cc}
    &
    \begin{array}{cccccccc}{}^1 & {}^3 & {}^2 & {}^6 & {}^8 & {}^4 & {}^5 & {}^7
    \end{array}
  \\
  \begin{array}{r}  {}_2\\{}_8\\{}_1\\{}_3\\{}_4\\{}_5\\{}_6\\{}_7
  \end{array}
  & \left(\begin{array}{cccccccc}
  1& & & & & & &  \\
   &1& & & & & &  \\
   & &1&1&1& & &  \\
   & &1&1&1& & &  \\
   & &1&1&1& & &  \\
   & & & & &1&1&1 \\
   & & & & &1&1&1 \\
   & & & & &1&1&1
  \end{array}\right)
\end{array}$
}\\

respectively, where the indexes on the top and on the left represent the original row numbers.
\end{remark}

\textbf{Step 3:}
We next choose $i=\alpha^{-1}(l)$. By Observation~\ref{obs2}, we must choose $i$ from $R_k(L,l)$. We have
\begin{equation}\label{eq3:3}
    \alpha(i)=l.
 \end{equation}
We compute $\sigma_{i,k}$. By (\ref{eq3:2}) and (\ref{eq3:3}), equation (\ref{eq3:1}) can be rewritten as
\begin{equation}\label{eq3:4}
\sigma_l=\alpha\pi_j^{-1}\sigma_{i,k}\pi_j\alpha^{-1}.
\end{equation}

\textbf{Step 4:}
We next choose a column $\pi_j$ and compute $\sigma_{i,j,k} := \pi_j^{-1}\sigma_{i,k}\pi_j$. We have, by (\ref{eq3:4}),
\begin{equation}\label{eq3:5}
\sigma_l=\alpha\sigma_{i,j,k}\alpha^{-1}.
\end{equation}

\textbf{Step 5:}
We manipulate the matrix $T(L,k)$ in the following way:
\begin{enumerate}
  \item We rearrange the rows by the cycle representation of $\sigma_{i,j,k}$ while keeping the original indexes. That is, for each cycle of $\sigma_{i,j,k}$, the rows indexed by its elements will be adjacent and the order in which the rows will be arranged will be the same as the order of the corresponding indexes within the cycle.
  \item In a similar manner, we rearrange the columns by the cycle representation of $\sigma_{l}$ while keeping the original indexes.
  \item We switch all the 1's in the row indexed by $i$, except for the 1 in the column indexed by $l$ and all the 1's in the column indexed by $l$, except for the 1 in the row indexed by $i$ (by (\ref{eq3:3})). In the row indexed by 1 we switch all the 1's that are not in columns indexed by elements of $\Delta(L)$ (by Observation~\ref{obs0}). In the row indexed by $k$ we switch all the 1's, except for the 1 in the column indexed by $1$ and all the 1's in the column indexed by $1$, except for the 1 in the row indexed by $k$ (by (\ref{eq3:2})).
  \item We subdivide the matrix into blocks according to the cycles of $\sigma_{i,j,k}$ and $\sigma_{l}$ (that is, the rows of each block are indexed by the numbers in a cycle of $\sigma_{i,j,k}$ and the columns of each block are indexed by the numbers in a cycle of $\sigma_{l}$). Then, switch all the 1's that appear in non-square blocks.
\end{enumerate}
The resulting matrix will be denoted $T^l(L,i,j,k)$.

\textbf{Step 6:}
This step constitutes most of the work. We shall use the following definition and proposition:
\begin{definition}
A \emph{shifted diagonal} in a square matrix $A=(a_{ij})_{i,j=1}^{m}$ of order $m$ is a set of cells $\{a_{i,t+i}\}_{i=1}^{m}$, for some $t\in\{1,\ldots,m\}$, where the indexes $t+i$ are taken modulo $m$ plus 1. Here is an illustration of a shifted diagonal:
\begin{equation*}
\left(\begin{array}{cc}
 &
\begin{array}{cccc}
* &  &  &  \\
 & * &  &  \\
 & & \ddots & \\
 & & & *
 \end{array}
 \\
\begin{array}{ccc}
* &  &   \\
 & \ddots & \\
 & & *
 \end{array}
 &
 \end{array}
 \right)
\end{equation*}
We shall use the term \emph{block shifted diagonal} for a shifted diagonal inside a square block of a given matrix.
\end{definition}

\begin{proposition}\label{prop3:1}
For any $\Theta_{\alpha,j}\in \A(L)$ there exist $k\in \Delta(L)$ and $i\in R_k(L,l)$ so that $\alpha$ is defined by an all-1's generalized diagonal $D_\alpha$ of $T^l(L,i,j,k)$, which satisfies the following property:
\begin{enumerate}
  \item [(*)] For any block $M$ in the subdivision of $T^l(L,i,j,k)$ (described in Step 5(4) above), either $D_\alpha\cap M =\emptyset$ or $D_\alpha\cap M$ is a shifted diagonal of $M$.
\end{enumerate}
\end{proposition}

\begin{proof}
Since $\alpha$ is defined by a generalized diagonal of $T^l(L,k)$, it still corresponds to a generalized diagonal after the rearrangements described in Step 5(1) and Step 5(2).
By (\ref{eq3:5}) and the identity
 \begin{equation}\label{eq3:6}
    \alpha(a_1,a_2,\ldots,a_t)\alpha^{-1}=(\alpha(a_1),\alpha(a_2),\ldots,\alpha(a_t)),
 \end{equation}
for any number $s\in[n]$, $s$ and $\alpha(s)$ appear in cycles of the same size of $\sigma_{i,j,k}$ and $\sigma_l$, respectively. Thus, by the rearrangement described in Step 5(1) and Step 5(2) and the subdivision described in Step 5(4), the intersection of $D_\alpha$ with any non-square block must be empty, and thus we can switch the 1's in these blocks.

Let $C$ be a cycle in the cycle representation of $\sigma_{i,j,k}$. By (\ref{eq3:6}), all the numbers appearing in $C$ are mapped by $\alpha$ to a set of numbers that form a cycle $C'$ in the cycle representation of $\sigma_l$.
Let $M_{C,C'}$ be the block in the subdivision of $T^l(L,i,j,k)$ whose rows are indexed by the elements of $C$ and whose columns are indexed the elements of $C'$. Thus, $D_\alpha\cap M_{C,C'}$ is a generalized diagonal of $M$.
Suppose $t$ succeeds $s$ in the cycle $C$ (this means that either $t$ appears immediately to the right of $s$, or $s$ is the rightmost element in $C$ and $t$ is the leftmost one), then by (\ref{eq3:6}), $\alpha(t)$ succeeds $\alpha(s)$ in $C'$. Thus, the 1's in the entries $(s,\alpha(s))$ and $(t,\alpha(t))$ in $T^l(L,i,j,k)$ are positioned in the block $M_{C,C'}$ in one of the four ways described in Figure~\ref{fig1}. In any case, they are part of a shifted diagonal of $M$.

\begin{figure}[h!]
  \centering
  \subfigure[]{\label{fig1a}\includegraphics[scale=0.2]{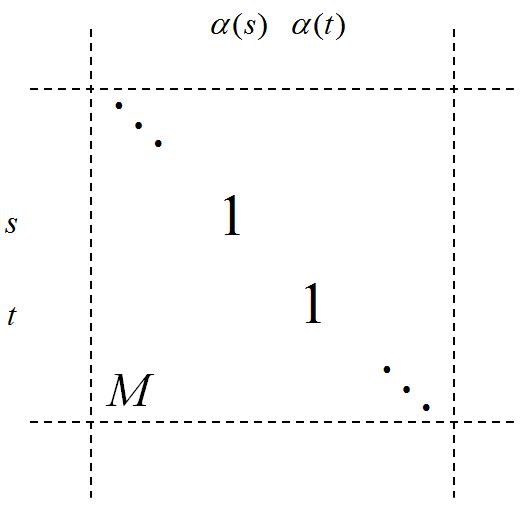}}
  \subfigure[]{\label{fig1b}\includegraphics[scale=0.2]{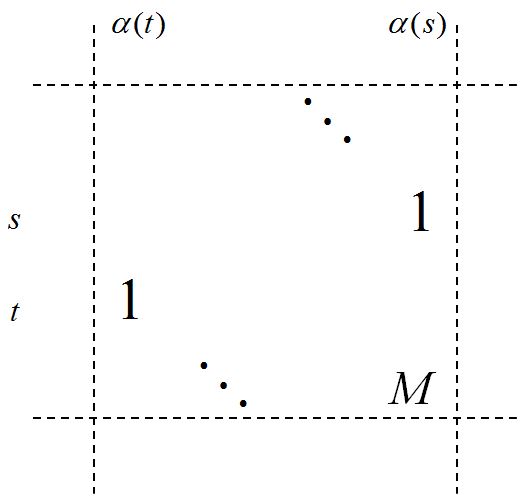}}
  \subfigure[]{\label{fig1c}\includegraphics[scale=0.2]{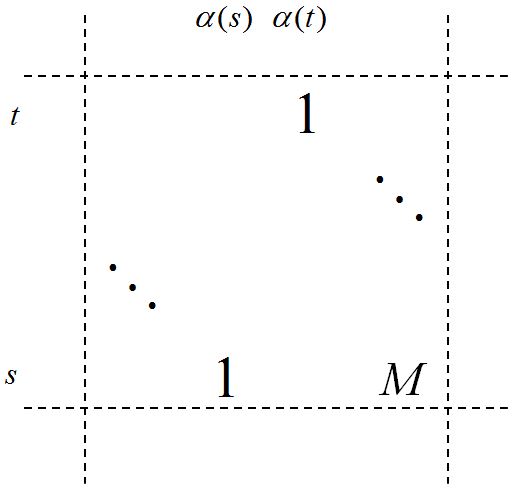}}
  \subfigure[]{\label{fig1d}\includegraphics[scale=0.2]{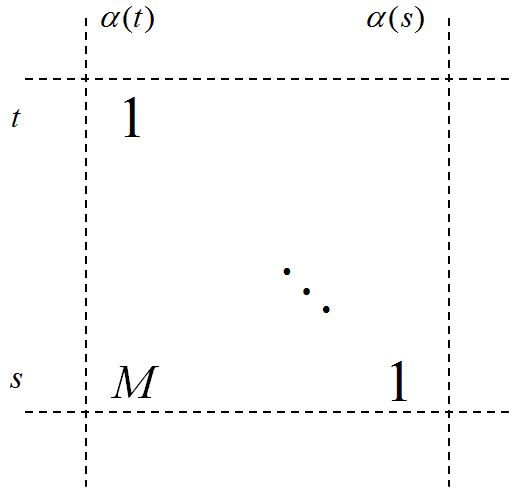}}
  \caption{ }
  \label{fig1}
\end{figure}

\end{proof}

The problem of finding $\A(L)$ is now reduced to finding the all-1's shifted diagonals in the square blocks of $T^l(L,i,j,k)$, for each  $i,j$ and $k$. In order to do this we don't have to iterate over all $\nu(L)\times \nu(L)$ blocks. Instead, we can iterate over all $\nu(L)$ cycles of $\sigma_l$.
For each such cycle $C$, we look at the columns indexed by its members, and consider the column with the least number of 1's. For each 1 in this column, we check whether it is part of an all-1's shifted diagonal in its block.
If in the columns indexed by some cycle of $\sigma_l$, there is not a single all-1's block shifted diagonal, then, clearly, there will be no all-1's generalized diagonal that satisfies the property (*) in $T^l(L,i,j,k)$. Hence, there is no need to check the rest of the cycles, and consequently, we can skip the rest of the steps in the algorithm and move on to the next matrix $T^l(L,i,j,k)$.

\textbf{Step 7:} After finding the all-1's block shifted diagonals in $T^l(L,i,j,k)$ (assuming that each cycle of $\sigma_l$ defines at least one) we combine them into generalized diagonals of $T^l(L,i,j,k)$. This will yield isotopisms $\Theta_{\alpha,j}$ that fix the row $\sigma_l$ (and, of course, the row $\sigma_1$, since all isotopisms of the form $\Theta_{\alpha,j}$ fix the first row).

\textbf{Step 8:} We use (\ref{eq6:01}) to check whether these $\Theta_{\alpha,j}$'s fix the other $n-2$ rows.

\begin{remark}
All the elements of $\A(L)$ constructed by this algorithm are distinct. To see this, first note that any two autotopisms $\Theta_{\alpha,j}$ and $\Theta_{\alpha',j}$ arising from the same matrix $T^l(L,i,j,k)$ are distinct, since they correspond to different diagonals of $T^l(L,i,j,k)$, so $\alpha\ne\alpha'$. Now, suppose $\Theta_{\alpha,j}$ and $\Theta_{\alpha',j'}$ arise from $T^l(L,i,j,k)$ and $T^l(L,i',j',k')$, respectively. If $j\ne j'$, then the autotopisms are clearly distinct. If either $i\ne i'$ or $k\ne k'$, then $\alpha\ne\alpha'$ by (\ref{eq3:2}) and (\ref{eq3:3}).
\end{remark}

\begin{remark}\label{rem4}
The algorithm was tested on 20,000 randomly generated Latin squares (by Jacobson Matthews method \cite{JM96}) of each of the orders 10, 15, 20, 25, and 30. The running times on an Intel Core i7 processor were 1.2, 2.1, 3.1, 4.3, and 5.7 seconds, respectively. The algorithm was also tested on an assortment of Latin squares with relatively small nontrivial autotopy group, produced by a program written for \cite{Wan04}. For 5000 Latin squares of each of the orders 10, 15, 20, 25, and 30, the times in seconds were 2.7, 3.7, 8.0, 15.4, and 23.4. It is worth noting the slow rate of growth as the order $n$ increases.
\end{remark}

\begin{example}\label{ex5}
Consider the Latin square in Example~\ref{ex0:1}.  We take $l=7$. Thus, the columns in any $T^7(L,i,j,k)$ are arranged by the cycles of $\sigma_7=(1,7,4,8,5)(2,6,3)$. We illustrate three different scenarios:
 \begin{enumerate}
   \item A candidate $\Theta_{\alpha,j}$  that is not an autotopism: Let $k=1$, $i=5$ and $j=6$. So, $\alpha(1)=1$ and $\alpha(5)=7$ (by (\ref{eq3:2}) and (\ref{eq3:3})). We have $\sigma_{5,1}=\sigma_{5}= (1,5,4)(2,7,3,6,8)$ (see (\ref{eq:cs1})) and $\sigma_{5,6,1}=\pi_6^{-1}\sigma_{5,1}\pi_6=(1,5,4,2,7)(3,8,6)$. After rearranging the rows and columns of $T(L,1)$ (in Remark~\ref{rem2}) by the cycles of $\sigma_{5,6,1}$ and $\sigma_7$ respectively, and deleting 1's according to Step 5(3) and Step 5(4), we obtain the following matrix $T^7(L,5,6,1)$:
       \begin{equation*}
\begin{array}{cc}
    &
    \begin{array}{cccccccc}{}^1 & {}^7 & {}^4 & {}^8 & {}^5 & {}^2 & {}^6 & {}^3
    \end{array}
  \\
  \begin{array}{r}  {}_1\\{}_5\\{}_4\\{}_2\\{}_7\\{}_3\\{}_8\\{}_6
  \end{array}
  & \left(\begin{array}{ccccc|ccc}
   \underline{1}& & & & & & &  \\
   &\underline{1}& & & & & &  \\
   & &\underline{1}& &1& & &  \\
   & & &\underline{1}& & & &  \\
   & &1& &\underline{1}& & &  \\
\hline
   & & & & & & &\underline{1} \\
   & & & & &\underline{1}&1&  \\
   & & & & &1&\underline{1}&
  \end{array}\right)
\end{array}
\end{equation*}
in which the only all 1's generalized diagonal satisfying the property (*) is underlined. This yields $\Theta_{\alpha,6}$ where  $\alpha=\left(\begin{array}{cccccccc}1&2&3&4&5&6&7&8\\1&8&3&4&7&6&5&2\end{array}\right)$. As explained above, $\Theta_{\alpha,6}$ was constructed to fix the first and seventh rows of $L$. However, $\Theta_{\alpha,6}\not\in\A(L)$ since it does not fix all the other rows.

   \item An autotopism: Let $k=2$, $i=5$ and $j=2$. So, $\alpha(2)=1$ and $\alpha(5)=7$. We have $\sigma_{5,2}=(2,5,3)(1,7,8,4,6)$ (see (\ref{eq:cs2})) and $\sigma_{5,2,2}=\pi_2^{-1}\sigma_{5,2}\pi_2=(1,4,8)(2,5,6,3,7)$. If we rearrange the rows of the matrix $T(L,2)$ (in Remark~\ref{rem2}) by the cycles in $\sigma_{5,2,2}$, and rearrange its columns by the cycles in $\sigma_7$ and delete 1's according to(3) and Step 5(4), we obtain the following matrix $T^7(L,5,2,2)$:
\begin{equation*}
\begin{array}{cc}
    &
    \begin{array}{cccccccc}{}^1 & {}^7 & {}^4 & {}^8 & {}^5 & {}^2 & {}^6 & {}^3
    \end{array}
  \\
  \begin{array}{r}  {}_2\\{}_5\\{}_6\\{}_3\\{}_7\\{}_1\\{}_4\\{}_8
  \end{array}
  & \left(\begin{array}{ccccc|ccc}
   1& & & & & & &  \\
   &1& & & & & &  \\
   & &1& &1& & &  \\
   & & &1& & & &  \\
   & & & &1& & &  \\
\hline
   & & & & &1& &  \\
   & & & & &1&1&  \\
   & & & & & & &1
  \end{array}\right)
\end{array}
\end{equation*}
We see that there is only one shifted diagonal in each square block (which happens to be the main diagonal). These yield an isotopism $\Theta_{\alpha,2}$, where
$\alpha=\left(\begin{array}{cccccccc}1&2&3&4&5&6&7&8\\2&1&8&6&7&4&5&3\end{array}\right)$.
After verifying that $\Theta_{\alpha,2}$ fixes all other rows we conclude that $\Theta_{\alpha,2}\in\A(L)$ (this is the only nontrivial autotopism of $L$).

   \item No candidates: Let $k=2$, $i=7$ and $j=7$. So, $\alpha(2)=1$ and $\alpha(7)=7$. We have $\sigma_{7,2}=(1,6,8)(2,7,3,5,4)$ (see (\ref{eq:cs2})) and $\sigma_{7,7,2}=\pi_7^{-1}\sigma_{7,2}\pi_7=(1,5,2,7,6)(3,4,8)$. If we rearrange the rows of $T(L,2)$  by the cycles in $\sigma_{7,7,2}$, and the columns by the cycles in $\sigma_7$, and delete 1's according to Step 5(3) and Step 5(4), we obtain the following matrix $T^7(L,7,7,2)$:
\begin{equation*}
\begin{array}{cc}
    &
    \begin{array}{cccccccc}{}^1 & {}^7 & {}^4 & {}^8 & {}^5 & {}^2 & {}^6 & {}^3
    \end{array}
  \\
  \begin{array}{r}  {}_1\\{}_5\\{}_2\\{}_7\\{}_6\\{}_3\\{}_4\\{}_8
  \end{array}
  & \left(\begin{array}{ccccc|ccc}
   & & & & & & &  \\
   & &1& &1& & &  \\
   1& & & & & & &  \\
   &1& & & & & &  \\
   & &1& &1& & &  \\
\hline
   & & & & &1&1&  \\
   & & & & &1&1&  \\
   & & & & & & &1
  \end{array}\right)
\end{array}
\end{equation*}
Since there are no 1's in the row indexed by 1, the matrix $T^7(L,7,7,2)$ doesn't yield any candidate to check in Step 8.
 \end{enumerate}
 \end{example}

\begin{remark}
For most Latin squares checked, the matrices $T^7(L,i,j,k)$ are as sparse as the ones illustrated above. Only for highly structures Latin squares we have dense matrices that produce a large number of diagonals satisfying Property (*). Such highly structured Latin squares are best handled by accelerated versions of ``nauty''.
\end{remark}
\section{A bound for $|\A(L)|$}
The algorithm described in the previous section yields a bound on the size of $\A(L)$. For convenience we introduce a new notation:

\begin{notation}\label{not6}
Let $L$ be a Latin square and let $C$ be a cycle in some row permutation of $L$. Denote $\lambda(L,C):=\min_{s\in C}\left(\lambda(L,s)\right)$.
\end{notation}

\begin{theorem}\label{thm4:1}
Let $L$ be a reduced Latin square of order $n$ and let $\sigma_l$ be a row in $L$ whose cycle representation contains $\nu(\sigma_l)=\nu(L)$ disjoint cycles. Suppose $\sigma_l=(C_1)(C_2)\ldots(C_{\nu(L)})$ is the cycle representation of $\sigma_l$ and assume $l\in C_1$. Then,
\begin{equation}\label{eq4:1}
    |\A(L)|\leq n \delta(L)\lambda(L,l)\prod_{i=2}^{\nu(L)} \lambda(L,C_i).
\end{equation}
\end{theorem}

\begin{proof}
There are $\delta(L)$ possible values of $k\in \Delta(L)$, $\lambda(L,l)$ possible values of $i\in R_k(L,l)$ (by Observation~\ref{obs2}) and $n$ possible columns $\pi_j$. Hence, there are $n \delta(L)\lambda(L,l)$ distinct matrices $T^l(L,i,j,k)$.
Consider one such matrix $T^l(L,i,j,k)$. By the discussion in the previous section, it defines all autotopisms $\Theta_{\alpha,j}$ of $L$ satisfying $\alpha(k)=1$ and $\alpha(i)=l$.
By Proposition~\ref{prop3:1}, the number of all-1's generalized diagonals in $T^l(L,i,j,k)$ satisfying the property (*) is an upper bound for the number of such autotopisms. By Property (*), any such diagonal is composed of all-1's block shifted diagonals in the subdivision of $T^l(L,i,j,k)$ defined in Step 5(4). The columns of each such block are indexed by a different cycle of $\sigma_l$.

Let $C$ be a cycle in the cycle representation of $\sigma_l$. The number of all-1's shifted diagonals in the columns indexed by the members of $C$ is at most $\lambda(L,C)$ (by Observations~\ref{obs1} and \ref{obs2}). Thus, the number of all-1's generalized diagonals in $T^l(L,i,j,k)$ satisfying the property (*) is at most $\prod_{i=1}^{\nu(L)} \lambda(L,C_i)$. By the deletions performed in Step(3), the columns defined by the cycle $C_1$ containing $l$ can have at most one all-1's shifted diagonal. Hence, we omit $\lambda(L,C_1)$ from the product.
\end{proof}

\begin{example}\label{ex3}
Consider the Latin square $L$ in Example~\ref{ex0:1}. The row $\sigma_7=(1,7,4,8,5)(2,6,3)$ has $\nu(L)=2$ cycles. Also, $\lambda(L,7)=3$ (there are 3 rows with the same cycle structure). Now, $7\in C_1=(1,7,4,8,5)$ and for $C_2=(2,6,3)$, $\lambda(L,C_2)=\lambda(L,3)=1$. Thus, by (\ref{eq4:1}), $\A(L)\le 8\cdot 2\cdot 3\cdot 1=48$.
\end{example}

The next proposition is closely related to the classical result that all principal loop isotopes of a loop agree if and only if the loop is a group (see, e.g., \cite{Bry1966}).

\begin{proposition}\label{prop4:1}
If equality holds in (\ref{eq4:1}) for a reduced Latin square $L$, then $L$ is the Cayley table of a group.
\end{proposition}
\begin{proof}
Let $\Theta_{\alpha,j}\in\A(L)$. Since equality holds in (\ref{eq4:1}), we must have $\Theta_{\alpha,k}\in\A(L)$ for all other $k\in[n]$. Thus, for any $k,l\in[n]$, we have,
\begin{equation*}
   \Phi=\Theta_{\alpha,k}^{-1}\Theta_{\alpha,l}=(1, \sigma^{-1}_{\alpha^{-1}(1)}\pi_k\pi_l^{-1}\sigma_{\alpha^{-1}(1)}, \pi_k\pi_l^{-1})\in\A(L).
\end{equation*}
Let $\sigma=\sigma_{\alpha^{-1}(1)}$ and $\beta=\sigma^{-1}\pi_k\pi_l^{-1}\sigma$.
We have, $\Phi=(1,\beta,\sigma\beta\sigma^{-1})\in\A(L)$. It follows from the convention at the end of Section~\ref{sec1} that after permuting the columns of $L$ by $\beta$, the first row becomes $\beta^{-1}$, and we have to apply $\beta$ on the symbols in order to transform the first row back into the identity permutation.
Hence, we have $\Phi=(1,\beta,\beta)$. By Proposition~\ref{prop1:1}(ii), $\beta=\pi_j^{-1}$ for some column $j$. We have $\pi_l\pi_k^{-1}=\pi_j$. Since this holds for any two columns $\pi_k$ and $\pi_l$, it follows that the columns of $L$ form a subgroup of $S_n$. This implies that $L$ is the Cayley table of some group (see \cite{al43}).
\end{proof}

The converse of Proposition~\ref{prop4:1} is not true. That is, if $L$ is the Cayley table of a group, equality in (\ref{eq4:1}) does not necessarily hold. However, for a large family of groups this is the case, showing that the bound in (\ref{eq4:1}) is tight:

\begin{proposition}
If $L$ is the Cayley table of a cyclic group, then equality holds in (\ref{eq4:1}).
\end{proposition}

\begin{proof}
Suppose $L$ is the Cayley table of the cyclic group $G$. By (\ref{eq1:0}), $|\A(G)|=n^2\phi(n)$ (Euler function). On the other hand, since $L$ is the Cayley table of a group, $\delta(L)=n$. The rows of $L$, viewed as permutations, form a subgroup of $S_n$ that is isomorphic to $G$, where the rows act by left composition. Since $G$ is cyclic, each of its generators corresponds to a row that is a single cycle. So, there are $\phi(n)$ single cycle rows in $L$. Let $\sigma_l$ be a single cycle row. By (\ref{eq4:1}), $|\A(L)|\le n^2\phi(n)$. Thus, we have equality.
\end{proof}
\section{Complexity Considerations}
The non-polynomial part of the algorithm is where the block shifted diagonals need to be combined to form generalized diagonals of $T^l(L,i,j,k)$ (Step 6). The complexity of this step is of order $n^{\nu(L)}$. Since $\nu(L)$ is not bounded, the algorithm, in the worst case, is of order $n^{n/2}$. However, as we shall see, in the vast majority of the cases, $\nu(L)$ is bounded and low, and for Latin squares that are not highly structured, the number of block shifted diagonals that need to be combined is low, if they exist at all. If we bound $\nu(L)$, the complexity of the computation of $\A(L)$ is polynomial:

\begin{theorem}\label{thm5:1}
Let $k$ be a fixed positive integer. For the set of all reduced Latin squares $L$ that have a row or a column with at most $k$ cycles, $\A(L)$ satisfies
\begin{equation}\label{eq5:1}
    |\A(L)|\le n\delta(n)\lambda(L)^k,
\end{equation}
and it can be computed in polynomial time in $n$.
\end{theorem}
\begin{proof}
Without loss of generality we may assume that $L$ contains a row $\sigma_l$ that decomposes into at most $k$ cycles. Thus, $\nu(L)\le k$ (if it is a column we take $L^T$. The relation between $\A(L)$ and $\A(L^T)$ is obvious).
The computation of each one of the matrices $T^l(L,i,j,k)$ is of polynomial complexity, since only elementary operations with permutations in $S_n$ are involved, and there are $n\delta(n)\lambda(L)$ such matrices.
As noted in Step 6, each cycle $C$ of $\sigma_l$, except for the cycle containing 1 and $l$, requires at most $\lambda(L,C)\le\lambda(L)$ iterations in order to find its corresponding all-1's block shifted diagonals.
Thus, the process of finding all the all 1's generalized diagonals satisfying the property (*) requires at most $\lambda(L)^{k-1}$ iterations. Thus, (\ref{eq5:1}) follows. Since $\lambda(L)<n$ and $\delta(n)\le n$, the computation of $\A(L)$ is polynomial in $n$.
\end{proof}
Since many Latin squares have a line that is a single cycle, we single out this case:
\begin{corollary}\label{cor4:1}
Let $L$ be a reduced Latin square of order $n$ such that at least one of its lines is a single cycle, then
\begin{equation*}
    |\A(L)|\le n^2(n-1),
\end{equation*}
\end{corollary}

In order to estimate the proportion of reduced Latin squares having a line that decomposes into at most $k$ cycles, we assume that there is weak dependence between the cycle structures of the lines. This has not been proved but seems to be the case, as indicated by the experiment described in the next paragraph.

Let $P_k(n)$ be the probability that a randomly chosen reduced Latin square has a line with at most $k$ cycles. We shall try to approximate $P_1(n)$.
There are approximately $n!/e$ derangements (permutations without fixed points) of order $n$,  among which $(n-1)!$ are single cycles.
Thus, the probability of each line, that is not the first row or the first column, of being a single cycle is approximately $e/n$. Assuming the weak dependence mentioned above, the probability that none of the rows $2,\ldots,n$ and columns $2,\ldots,n$ is a single cycle is approximately $(1-(e/n))^{2(n-1)}\approx e^{-2e}\approx0.0044$ for large $n$.
Thus, $P_1(n)\approx0.9956$. Indeed, 99,580 of 100,000 randomly generated reduced Latin squares of order 20 had at least one single cycle line. The probability that a random derangement of order 20 has at most 2 cycles is $\approx 0.475$. Thus, assuming the weak dependence mentioned above, we have $P_2(20)\approx 0.999999999977$. (The number of $n$-derangements with exactly $k$ cycles, denoted $d(n,k)$, can be computed using the recurrence relations $d(n,k)=(n-1)d(n-1,k)+(n-1)d(n-2,k-1)$ when $k\le n/2$, and $d(n,k)=0$ when $k>n/2$, with base case $d(2,1)=1$.)

Given two rows, indexed $r$ and $s$, of a Latin square, Cavenagh, Greenhill and Wanless \cite{Cav08} defined $\sigma_{s,r}$ as the permutation that transforms $\sigma_r$ to $\sigma_s$. That is, $\sigma_{s,r}=\sigma_s\sigma_r^{-1}$. It is said that the rows $r$ and $s$ \emph{have relative cycle structure} $c$ if the cycle structure of $\sigma_{s,r}$ is $c$ (this definition is symmetric in $r$ and $s$ since the cycle structure of a permutation and its inverse are the same). We have:

\begin{theorem}\label{thm5:2}
Let $k$ be a fixed positive integer. For the set of all reduced Latin squares $L$ that have two rows or two columns with relative cycle structure containing at most $k$ cycles, $\A(L)$ satisfies
\begin{equation*}
    |\A(L)|\le n^2(n-1)^k,
\end{equation*}
and it can be computed in polynomial time in $n$.
\end{theorem}
\begin{proof}
Let $L$ be a reduced Latin and assume $\sigma_{s,r}$ has at most $k$ cycles. Let $\alpha$ be a permutation satisfying $\alpha(1)=r$ and let $\Theta_{\alpha,j}$ be an isotopism defined by some column $\pi_j$ of $L$ as in Proposition~\ref{prop1:1}. Let $i=\alpha(s)$ and let $L'=\Theta_{\alpha,j}(L)$. By (\ref{eq6:01}), $L'$ has a row whose cycle structure has at most $k$ cycles. Since $\A(L')$ can be computed in polynomial time, so can $\A(L)$. By (\ref{eq5:1}), $|\A(L')|=|\A(L)|\le n\delta(L)\lambda(L)^k\le n^2(n-1)^k$.
\end{proof}

Let $\mathbb{P}_k(n)$ be the probability that a randomly chosen reduced Latin square has either two rows or two columns with a relative cycle structure containing at most $k$ cycles. We try to approximate $\mathbb{P}_1(n)$. It was conjectured in \cite{Cav08} that $\sigma_{s,r}$ shares the asymptotic distribution of a random derangement. This means that the probability that $\sigma_{s,r}$ is a single cycle would be approximately $e/n$. Assuming this and the weak dependence mentioned in the paragraph following Corollary~\ref{cor4:1} we obtain that $\mathbb{P}_1(n)\approx 1- (1-e/n)^{n(n-1)}$, which tends very quickly to 1. This agrees with the result of McKay and Wanless \cite{MckayWan05} who proved that the proportion of order $n$ Latin squares which have a non-trivial symmetry tends very quickly to zero. However, if the conjecture in \cite{Cav08} holds, then we have a slightly different statement, namely, that the proportion of Latin squares for which the computation of $\A(L)$ is polynomial tends very quickly to 1.

\section*{Acknowledgment}
I thank Ian Wanless for providing the programs for generating Latin squares for the experiments mentioned in Remark~\ref{rem4}. I thank an anonymous referee for a thorough reading of the manuscripts, for introducing many improvements, and for bringing the Bryant-Schneider paper to my attention.

\begin{thebibliography}{10}

\bibitem{al43}
A.~A.~Albert, \emph{Quasigroups {I}}, Transactions of the American Mathematical
  Society \textbf{54} (1943), no.~3, 507--519.

\bibitem{Bailey82}
R.~A. Bailey, \emph{{L}atin squares with highly transitive automorphism
  groups}, Journal of the Australian Mathematical Society \textbf{33} (1982),
  no.~1, 18--22.

\bibitem{Bro2012}
J.~Browning, D.~S. Stones, and I.~M. Wanless, \emph{Bounds on the number of
  autotopisms and subsquares of a {L}atin square}, Combinatorica, to appear.

\bibitem{Bry1966}
B.~F.~Bryant and H.~Schneider, \emph{Principal loop-isotopes of quasigroups}, Canad. J. Math, \textbf{18} (1966), 120--125.

\bibitem{Cav08}
N.~J. Cavenagh, C.~Greenhill, and I.~M. Wanless, \emph{The cycle structure of
  two rows in a random {L}atin square}, Random Structures and Algorithms
  \textbf{33} (2008), 286--309.

\bibitem{Denes74}
J.~D\'{e}nes and A.~D. Keedwell, \emph{Latin squares and their applications},
  Academic Press, New York, 1974.

\bibitem{Fal09}
R.~M. Falc\'{o}n, \emph{Cycle structures of autotopisms of the {L}atin squares
  of order up to 11}, Ars Combin., to appear.

\bibitem{galuszka2008codes}
J.~Ga{\l}uszka, \emph{Groupoids with quasigroup and {L}atin square properties},
  Discrete Mathematics \textbf{308} (2008), no.~24, 6414--6425.

\bibitem{JM96}
M.~T. Jacobson and P.~Matthews, \emph{Generating uniformly distributed random
  {L}atin squares}, J. Combinatorial Designs \textbf{4} (1996), no.~6,
  405--437.

\bibitem{Janssen95}
J.~C.~M. Janssen, \emph{On even and odd {L}atin squares}, Journal of
  Combinatorial Theory A \textbf{69} (1995), 173--181.

\bibitem{kot12}
D.~Kotlar, \emph{Parity types, cycle structures and autotopisms of latin
  squares}, Electronic Journal of Combinatorics \textbf{19} (2012), no.~3, P10.

\bibitem{McKay}
B.~D. McKay, \emph{nauty {U}ser's {G}uide ({V}ersion 1.5)}, Computer Science
  Technical Report TR-CS-90-02, Australian National University, 1990.

\bibitem{MckayMM07}
B.~D. McKay, A.~Meynert, and W.~Myrvold, \emph{Small {L}atin squares,
  quasigroups, and loops}, J. Combinatorial Designs \textbf{15} (2007),
  98--119.

\bibitem{MckayWan05}
B.~D. McKay and I.~M. Wanless, \emph{On the number of {L}atin squares}, Annals
  of Combinatorics \textbf{9} (2005), 335--344.

\bibitem{Sade68}
A.~Sade, \emph{Autotopies des quasigroupes et des systemes associatifs},
  Archivum Mathematicum \textbf{4} (1968), no.~1, 1--23.

\bibitem{Sch30}
E.~Sch\"{o}nhardt, \emph{\"{U}ber lateinische {Q}uadrate und {U}nionen}, J.
  Reine Angew. Math. \textbf{163} (1930), 183--229.

\bibitem{Stones11}
D.~S. Stones, P.~Vojt\v{e}chovsk\'{y}, and I.~Wanless, \emph{Cycle structure of
  autotopisms of quasigroups and {L}atin squares}, Journal of Combinatorial
  Designs \textbf{20} (2012), no.~5, 227--263.

\bibitem{StWan12}
D.~S. Stones and I.~M. Wanless, \emph{How not to prove the {A}lon-{T}arsi
  conjecture}, Nagoya Math J. \textbf{205} (2012), 1--24.

\bibitem{Wan04}
I.~M. Wanless, \emph{Diagonally cyclic {L}atin squares}, European J.\ Combin. \textbf{25} (2004), 393--413.

\bibitem{wan05}
I.~M. Wanless, \emph{Atomic {L}atin squares based on cyclotomic
  orthomorphisms}, Electronic Journal of Combinatorics \textbf{12} (2005).

\end{thebibliography}
\end{document}